\documentclass[10pt]{article}
\usepackage[english]{babel}
\usepackage[utf8]{inputenc}
\usepackage[T1]{fontenc}
\usepackage{indentfirst}
\usepackage[a4paper]{geometry}
\usepackage{amsmath,amssymb}
\usepackage{amsthm}
\usepackage{mathrsfs}
\usepackage{mathtools}
\usepackage{dsfont}
\usepackage{enumitem}

\usepackage{comment}

\makeatletter
\newenvironment{chapabstract}{%
    \begin{center}%
      \bfseries Abstract
    \end{center}}%
   {\par}
\makeatother

\theoremstyle{definition}
\newtheorem{defn}{Definition}[section]
\newtheorem{exam}[defn]{Example}
\newtheorem{rem}[defn]{Remark}
\theoremstyle{plain}
\newtheorem{lem}[defn]{Lemma}
\newtheorem{thm}[defn]{Theorem}

\newtheorem{pr}[defn]{Proposition}

\DeclareMathOperator*{\esssup}{ess\,sup}

\newcommand{\e}{\mathbb{E}}
\newcommand{\rset}{\mathbb{R}}
\begin{document}
\title{Particles Systems for Mean Reflected BSDEs}
\author{Philippe Briand\thanks{Univ. Grenoble Alpes, Univ. Savoie Mont Blanc, CNRS, LAMA, 73000 Chambéry, France
(philippe.briand@univ-smb.fr)}  
\and H\'el\`ene Hibon\thanks{Univ. Rennes, CNRS, IRMAR - UMR 6625, F-35000 Rennes, France (helene.hibon@univ-rennes1.fr), partially supported by Lebesgue center of mathematics ("Investissements d'avenir"
program - ANR-11-LABX-0020-01 and by ANR-15-CE05-0024-02.}}
\date{}

\maketitle
\begin{chapabstract}
In this paper, we deal with Reflected Backward Stochastic Differential Equations for which the constraint is not on the paths of the solution but on its law as introduced by Briand, Elie and Hu in \cite{MRBSDE}.
We extend the recent work \cite{MRSDE} of Briand, Chaudru de Raynal, Guillin and Labart on the chaos propagation for mean reflected SDEs to the backward framework. When the driver does not depend on $z$, we are able to treat general reflexions for the particles system. We consider linear reflexion when the driver depends also on $z$.
In both cases, we get the rate of convergence of the particles system towards the square integrable deterministic flat solution to the mean reflected BSDE.
\end{chapabstract}


\section{Introduction}

Since their introduction by Pardoux and Peng \cite{PP90} in the 90's, BSDEs have been much studied. They are particularly useful for formulating problems in mathematical finance. In 1997, El Karoui, Kapoudjian, Pardoux, Peng, and Quenez \cite{ElKar} developed the notion of reflected BSDEs
\[ Y_t=\xi+\int_t^T{f(u,Y_u,Z_u) \,du}-\int_t^T{Z_u \,dB_u}+K_T-K_t \quad \forall t\in[0,T] \]
related to obstacle and optimal stopping problems. They formulate therefore their constraint as $Y_t\geq L_t$. More recently, Hu and Tang \cite{HT10} have studied multi-dimensional BSDEs with oblique reflection and their application to optimal switching problems. We refer to the introduction of Briand, Elie and Hu \cite{MRBSDE} for further motivations and references for considering reflected BSDEs.

As in \cite{MRBSDE}, we are concerned here by mean reflected BSDEs (MRBSDEs in short), which are reflected BSDEs with a constraint on the law of the process $Y$ rather than on its paths :
\[\begin{dcases} Y_t=\xi+\int_t^T{f(u,Y_u,Z_u) \,du}-\int_t^T{Z_u \,dB_u}+K_T-K_t, \\ \mathbb{E}[h(Y_t)]\geq 0, \end{dcases} \quad \forall t\in[0,T]\]
with deterministic $K$ and with the Skorokhod condition "$\int_0^T{\mathbb{E}[h(Y_t)] \,dK_t}=0$" \,that allows us to qualify the solution as "flat" when satisfied. Such a model is related to risk measures and acceptance sets that correspond to each other via $\mathcal{A}_{\rho}=\{X:\rho(X)\leq 0\}\,,\, \rho_{\mathcal{A}}(X)=\inf\{r\in\mathbb{R}: r+X\in\mathcal{A}\}$. In the case of an acceptance set is of the form $\mathcal{A}_{\rho}=\{X: \mathbb{E}[h(X)]\geq 0\}$ with $h$ being roughly speaking a utility function, solving the mean reflected BSDE means that $Y_t$ has to be an acceptable position at each $t$. The Value at Risk $\text{VAR}_{\alpha}$ is a typical example, see \cite{Rm1} and \cite{Rm2} for its definition and an overview on coherent and convex risk measures.

We extend in this paper the recent work of Briand, Chaudru de Raynal, Guillin and Labart on the propagation of chaos for mean reflected SDEs \cite{MRSDE} to the backward framework. Their study allows to approximate the solution of mean reflected SDEs by an interacting particles system. The interaction consists in a trajectory reflection and such reflected BSDEs have been widely studied.  Moreover, it is not possible to numerically compute the solution of a mean reflected solutions while several algorithms exist for particles systems based on the empirical distribution.  For more details on propagation of chaos, we refer to Sznitman's notes \cite{Chaos}.

\smallskip

The paper is organized as follows. In Section 2, we present our framework and we recall some results for Mean Reflected BSDEs of \cite{MRBSDE} and we introduce our particles system which turns to be a multidimensional  reflected BSDE in a set which is not necessarily convex. In Section 3, we prove that BSDEs coming from these particles systems have a unique solution in the case where the driver does depend on the variables. This the starting point of the other results. Section 4 contains our main result. In the case where the driver does not depend on $z$, we construct a solution to our particles system and we prove that this system converges to the solution of the mean reflected BSDE. We give also the rate of convergence of the propagation of chaos. Finally, in the last section, we consider the case of general drivers, depending on both $y$ and $z$, for which we manage to treat only linear reflexions.

\smallskip

Let us finish this introduction by giving some notations.

\paragraph*{Notations.}

We will work throughout this paper with the Euclidean norm $|.|$\, and denote for $p\geq 1$

\smallskip

$\mathscr{S}^p$ the set of adapted continuous processes $Y$ on $[0,T]$ such that $\|Y\|_{\mathscr{S}^p}:=\mathbb{E}\left[ \underset{0\leq t\leq T]}{\sup}|Y_t|^p\right]^{1/p}<\infty$

\smallskip

$\mathscr{A}^2$ the closed subset of $\mathscr{S}^2$ consisting of non-decreasing processes starting from $0$.

\smallskip

$\mathscr{M}^p$ the set of predictable processes $Z$ such that $\displaystyle \|Z\|_{\mathscr{M}^p}:=\mathbb{E}\left[\left|\int_0^T{|Z_u|^2 du}\right|^{p/2}\right]^{1/p}<\infty$

\section{Framework}

\subsection{MRBSDE}

Let us first recall some results from \cite{MRBSDE} on Mean Reflected BSDEs (MRBSDEs in short).

\noindent Consider the Mean Reflected BSDE on $(\Omega,\mathscr{F},\mathbb{P})$ endowed with a standard Brownian motion $B=\left(B_t\right)_{0\leq t\leq T}$ of which we denote $\left\{\mathscr{F}_t \,,\, 0\leq t\leq T\right\}$ the augmented natural filtration
\begin{equation}\label{eqMRBSDE}
\begin{dcases} Y_t=\xi+\int_t^T{f(u,Y_u,Z_u) \,du}-\int_t^T{Z_u \,dB_u}+K_T-K_t, \\ \mathbb{E}[h(Y_t)]\geq 0, \end{dcases} \quad 0\leq t\leq T
\end{equation}
and the following set of assumptions

\begin{itemize}[label=($H_\xi$)]
\item The terminal condition $\xi$ is a square integrable $\mathscr{F}_T$-measurable random variable and
\begin{center}
$\mathbb{E}[h(\xi)]\geq 0$.
\end{center}
\end{itemize}

\begin{itemize}[label=($H_f$)]
\item The driver $f:\Omega\times[0,T]\times\mathbb{R}\times\mathbb{R} \rightarrow \mathbb{R}$ is a measurable map, $f(.,0,0)\in \mathscr{M}^2$ and there exists $\lambda \geq 0$ such that $\mathbb{P}$-a.s
\begin{center}
$|f(t,y,z)-f(t,p,q)|\leq \lambda \left(|y-p|+|z-q|\right) \quad \forall t\in[0,T] ,\, \forall y,p,z,q\in\mathbb{R}$.
\end{center}
\end{itemize}

\begin{itemize}[label=($H_h$)]
\item The function $h$ is increasing and bi-Lipschitz: there exist $0<m\leq M$ such that
\begin{center}
$m|x-y|\leq |h(x)-h(y)|\leq M|x-y| \quad \forall x,y\in\mathbb{R}$.
\end{center}
\end{itemize}

\begin{defn}
A \textit{square integrable} solution to the MRBSDE \eqref{eqMRBSDE} is a triple of processes $(Y,Z,K)$ in the space $\mathscr{S}^2\times \mathscr{M}^2\times \mathscr{A}^2$ satisfying the equation together with the constraint. A solution is said to be \textit{flat} if moreover $K$ increases only when needed$,$ i.e we have
\begin{center}
$\displaystyle\int_0^T{\mathbb{E}[h(Y_t)] \,dK_t}=0$
\end{center}
By a \textit{deterministic} solution$,$ we mean a solution for which the process $K$ is deterministic.
\end{defn}

\begin{thm}
Suppose that the parameters $\xi,f$ and $h$ satisfy assumptions $(H_{\xi}), (H_f)$ and $(H_h)$. Then the MRBSDE \eqref{eqMRBSDE} admits a unique square integrable deterministic flat solution.
\end{thm}

Let us recall that, in the constant driver case, the deterministic process $K$ is given by the formula
\[
R_t:=K_T-K_t=\underset{s\geq t}{\sup} \inf\left\{x\geq 0 : \mathbb{E}\left[h\left(x+\mathbb{E}\left[\xi+\int_s^T{f_u\,du}\,\middle|\,\mathscr{F}_s\right]\right)\right]\geq 0\right\}:=\underset{s\geq t}{\sup}\, \psi_s.
\]
The non-constant driver case is obtained thanks to a fixed point argument.

\subsection{Interacting particle system}

Given $N\in\mathbb{N}^*,$ introduce $\{\xi^i\}_{1\leq i\leq N}, \{f^i\}_{1\leq i\leq N}$ and $\{B^i\}_{1\leq i\leq N}$ independent copies of $\xi, f$ and $B$. More precisely,  if $\xi = G\left(\{B_t\}_{0\leq t\leq T}\right)$ and $f(t,y,z)=F\left(t,\{B_{s\wedge t}\}_{0\leq s\leq T},y,z\right)$ for some measurable $G$ and $F$, we take
\begin{equation*}
	\xi^i = G\left(\{B^i_t\}_{0\leq t\leq T}\right), \qquad f^i(t,y,z) = F\left(t,\{B^i_{s\wedge t}\}_{0\leq s\leq T},y,z\right).
\end{equation*}

\noindent The augmented natural filtration of the family of Brownian motions $\{B^i\}_{1\leq i\leq N}$ is denoted $\mathscr{F}^{(N)}$. 

For all $1\leq i\leq N$, let us define  $\theta^i:=\xi^i+\psi_T^{(N)}$ where
\begin{equation*}
	\psi_T^{(N)}:=\inf\left\{ x\geq 0 : \frac{1}{N}\sum_{i=1}^N{h\left(x+\xi^i\right)}\geq 0 \right\},
\end{equation*}
and let us consider the following multidimensional reflected BSDE: 
\begin{equation}\label{eqPS}
\begin{dcases}{} Y_t^i=\theta^i+\int_t^T{f^i(u,Y_u^i,Z_u^{i,i}) \,du}-\int_t^T{\sum_{j=1}^N{Z_u^{i,j} \,dB_u^j}}+K_T^{(N)}-K_t^{(N)} \quad \forall\, 1\leq i\leq N,\\ \frac{1}{N}\sum_{i=1}^N{h(Y_t^i)}\geq 0, \end{dcases} \; 0\leq t\leq T.
\end{equation}
This equation is a multidimensional reflected BSDE in a possibly non convex domain, the domain being convex if and only if the function $h$ is concave.

\begin{rem}
	We add the term $\psi_T^{(N)}$ to each random variable $\xi^i$ to ensure that the condition is satisfied at the terminal time. Indeed, even though the expected value of $h(\xi)$ is positive, we do not have in general
	\begin{equation*}
		\frac{1}{N}\sum_{i=1}^N h\left(\xi^i\right) \geq 0, \quad \mathbb{P}-a.s.
	\end{equation*}
	However, by definition of $\psi_T^{(N)}$, we have
	\begin{equation*}
		\frac{1}{N}\sum_{i=1}^N{h\left(\theta^i\right)}\geq 0, \quad \mathbb{P}-a.s.
	\end{equation*}
\end{rem}

\begin{defn}
A solution $\left(\left\{Y^i,Z^i\right\}_{1\leq i\leq N},K^{(N)}\right)$ to \eqref{eqPS} is said to be \textit{flat} if the Skorokhod condition is satisfied namely
\begin{center}
$\displaystyle \int_0^T{\frac{1}{N}\sum_{i=1}^N{h(Y_t^i)}\; dK_t^{(N)}}=0$.
\end{center}
\end{defn}

\noindent The study of this equation will start with the constant driver case. We state the existence and uniqueness result$,$ needed to develop the fixed point argument for non-constant drivers $,$ but also some a priori estimate that we will use numerous times.

\section{The particle system with constant driver }\label{Tpswcd}

In this section, we consider the case where the driver does not depend on $(y,z)$. Equation \eqref{eqPS} rewrites
\begin{equation}\label{eqCPS}
\begin{dcases} Y_t^i=\theta^i+\int_t^T{f_u^i \,du}-\int_t^T{\sum_{j=1}^N{Z_u^{i,j} \,dB_u^j}}+K_T^{(N)}-K_t^{(N)}, \quad \forall\, 1\leq i\leq N,\\ \frac{1}{N}\sum_{i=1}^N{h(Y_t^i)}\geq 0, \end{dcases} \; 0\leq t\leq T,
\end{equation}
where we recall that, for all $1\leq i\leq N$, $\theta^i=\xi^i+\psi_T^{(N)}$ where
\begin{equation*}
	\psi_T^{(N)}=\inf\left\{ x\geq 0 : \frac{1}{N}\sum_{i=1}^N{h\left(x+\xi^i\right)}\geq 0 \right\}.
\end{equation*}

\noindent Assumption $(H_f)$ reduces in this case to
\begin{itemize}[label=($\tilde{H}_f$)]
\item The process $\{f_s\}_{0\leq s\leq T}$ is square integrable and progressively measurable i.e. $f\in\mathscr{M}^2$.
\end{itemize}

Before stating the result, let us introduce some further notations. Let us consider the progressively measurable process $\psi^{(N)}$ defined by
\[
\psi_t^{(N)}=\inf\left\{x\geq 0 : \frac{1}{N}\, \sum_{i=1}^N h\left(x+U^i_t\right)\geq 0\right\}, \quad 0\leq t\leq T,
\]
where we have set, for all $1\leq i\leq N$,
\[
U_t^i= \mathbb{E}\left[\xi^i+\int_t^T{f_u^i\, du}\,\middle|\,\mathscr{F}_t^{(N)}\right], \quad 0\leq t\leq T.
\]

\begin{thm}\label{CPS}
Assume $(H_{\xi}),(H_h)$ and $(\tilde{H}_f)$ are satisfied. Then the multidimensional reflected BSDE \eqref{eqCPS} admits a unique flat solution in the product space $\left(\mathscr{S}^2(\mathbb{R})\times \mathscr{M}^2(\mathbb{R}^N)\right)^N\times \mathscr{A}^2(\mathbb{R})$.

Moreover, we have, for $1\leq i\leq N$,
\begin{equation*}
	Y^i_t = U^i_t + S_t, \quad 0\leq t\leq T,
\end{equation*}
where $S$ is the Snell envelope of the process $\psi^{(N)}$.
\end{thm}

\begin{proof}
	Let us start by constructing a solution. Let us observe that the process $\psi^{(N)}$ can be written as
	\begin{equation*}
		\psi_t^{(N)} = L\left(U^1_t,\ldots,U^N_t\right), \quad 0\leq t\leq T,
	\end{equation*}
	where, for any $X=(X^1,\ldots,X^N)$ in $\rset^N$, 
	\[
	L(X)=L(X^1,\ldots,X^N) = \inf\left\{x\geq 0 : \frac{1}{N}\sum_{i=1}^N h\left(x+ X^i\right) \geq 0\right\}.
	\]

	As pointed out in \cite{MRBSDE} and \cite{MRSDE}, $L$ is Lipschitz continuous. More precisely,
	\begin{equation}\label{eq:Llip}
		|L(X)-L(Y)|\leq \dfrac{M}{m}\,\dfrac{1}{N}\,\sum_{j=1}^N |X^j-Y^j|.
	\end{equation}
	Indeed, since $h$ is bi-Lipschitz and increasing, we have
	\begin{align*}
		h\left(L(X)+\dfrac{M}{m}\,\dfrac{1}{N}\, \sum_{j=1}^N |X^j-Y^j|+Y^i\right)& \geq m\dfrac{M}{m}\dfrac{1}{N}\sum_{j=1}^N |X^j-Y^j|+h(L(X)+Y^i), \\
		& \geq \dfrac{M}{N}\, \sum_{j=1}^N |X^j-Y^j|+h(L(X)+X^i)-M|X^i-Y^i|.
	\end{align*}
	Summing these inequalities, we get, by definition of $L$,
	\begin{align*}
		\frac{1}{N}\, \sum_{i=1}^N h\left(L(X)+\dfrac{M}{m}\dfrac{1}{N}\sum_j|X^j-Y^j|+Y^i\right) \geq \frac{1}{N}\, \sum_{i=1}^N h(L(X)+X^i) \geq 0.
	\end{align*}
	Thus, using again the definition of $L$,
	 \begin{equation*}
	 	L(Y)\leq L(X)+\dfrac{M}{m}\,\dfrac{1}{N}\,\sum_{j=1}^N |X^j-Y^j|,
	 \end{equation*}
	and the result follows by symmetry. 

	As a byproduct, the process $\psi^{(N)}$ belongs to $\mathscr{S}^2$ and moreover, there exists a constant $C$ independent of $N$ such that :
	\begin{equation}\label{eq:majpsi}
		\mathbb{E}\left[\sup_{0\leq t\leq T} \left|\psi^{(N)}_t\right|^2\right] \leq C \left( 1+ \mathbb{E}\left[|\xi^2| + \int_0^T |f_s|^2 ds\right]\right).
	\end{equation}
	Indeed, let us set $x_0:=\inf\left\{x\geq 0 : h(x)\geq 0\right\}$ which is finite in view of the assumptions on $h$. We have
	\begin{equation*}
		|\psi_t^{(N)}| = \left|L(U_t^1,...,U_t^N)-L(0)+L(0)\right|\leq x_0 + \dfrac{M}{m}\dfrac{1}{N}\sum_j|U_t^j|
	\end{equation*}
	and the estimate follows from Doob and H\"older inequalities.

	Since $\psi^{(N)}$ is in $\mathscr{S}^2$, its Snell envelope $S$ exists and belongs to $\mathscr{S}^2$. In fact $S$ can be taken as a right continuous $\mathscr{F}^{(N)}$-supermartingale  of class (D). Its Doob-Meyer decomposition provides us the existence and uniqueness of $(K^{(N)},M^{(N)})$, square integrable, with $K^{(N)}$ a non-decreasing process starting from $0$ and $M^{(N)}$ a $\mathscr{F}^{(N)}$-martingale such that 
	\[
	S_t=M_t^{(N)}-K_t^{(N)}, \quad 0\leq t\leq T.
	\]
	Since $S_T=\psi_T^{(N)}$
	\begin{equation*}
		S_t=\mathbb{E}\left[M_T^{(N)}\,\middle|\,\mathscr{F}_t^{(N)}\right]-K_t^{(N)}=\mathbb{E}\left[\psi_T^{(N)} + K_T^{(N)}\,\middle|\,\mathscr{F}_t^{(N)}\right]-K_t^{(N)}.
	\end{equation*}
	We obtain moreover that $K_T^{(N)}$ is square integrable.

	Let us set, for $1\leq i\leq N$,
	\begin{equation*}
		Y_t^i=U_t^i + S_t, \quad 0\leq t\leq T.
	\end{equation*}
	We have, 
	\begin{align*}
		Y^i_t & =\mathbb{E}\left[\xi^i+\int_t^T{f_u^i\, du}\,\middle|\,\mathscr{F}_t^{(N)}\right]+\mathbb{E}\left[\psi_T^{(N)} + K_T^{(N)}\,\middle|\,\mathscr{F}_t^{(N)}\right]-K_t^{(N)} \\
		& =\mathbb{E}\left[\theta^i+\int_0^T{f_u^i\, du}+K_T^{(N)}\middle|\,\mathscr{F}_t^{(N)}\right]-\int_0^t{f_u^i\, du}-K_t^{(N)}.
	\end{align*}
	We can apply the representation theorem for $L^2$-martingales to write for some $Z^i$ in $\mathscr{M}^2(\mathbb{R}^N)$
	\[
	Y_t^i=\mathbb{E}\left[\theta^i+\int_0^T{f_u^i\, du}+K_T^{(N)}\right]+\int_0^t{\sum_{j=1}^N{Z_u^{i,j}\,dB_u^j}}-\int_0^t{f_u^i\, du}-K_t^{(N)}.
	\]
	We verify easily that $\left(\left\{Y^i,Z^i\right\}_{1\leq i \leq N},K^{(N)}\right)$ is a solution to \eqref{eqCPS}. 

	For the constraint, since $h$ is nondecreasing, we have, by definition of $\psi^{(N)}$, 
	\[ 
	\frac{1}{N}\,\sum_{i=1}^N h\left(Y_t^i\right) =  \frac{1}{N}\,\sum_{i=1}^N h\left(U_t^i+S_t\right) \geq  \frac{1}{N}\,\sum_{i=1}^N h\left(U_t^i+\psi_t^{(N)}\right) \geq 0.  \]

	It remains to prove that the Skorokhod condition is satisfied. Since $S$ is the Snell envelope of $\psi^{(N)}$ and $K^{(N)}$ is the associated nondecreasing process, $S_t=\psi_t^{(N)} \; dK_t^{(N)}$ almost everywhere. Let us observe modeover that
	\begin{align*}
		0=S_t \mathds{1}_{S_t=0}=\mathbb{E}\left[\left(\psi_T^{(N)} + K_T^{(N)}- K_t^{(N)}\right)\mathds{1}_{S_t=0} \,\middle|\, \mathscr{F}_t^{(N)}\right].
	\end{align*}
	Since $K^{(N)}$ is nondecreasing and $\psi^{(N)}$ nonnegative, we deduce $K^{(N)}$ is constant on $[t,T]$ on the set $\{S_t=0\}$. Thus, we have
	\begin{align*}
		\int_0^T \frac{1}{N}\sum_{i=1}^N h\left(Y_t^i\right) \, dK_t^{(N)}  & =\int_0^T \frac{1}{N}\sum_{i=1}^N h\left(U_t^i+S_t\right)\, \mathds{1}_{S_t>0}\, dK_t^{(N)},\\
	&=\int_0^T \frac{1}{N}\sum_{i=1}^N h\left(U_t^i+\psi^{(N)}_t\right)\, \mathds{1}_{\psi^{(N)}_t>0} dK_t^{(N)} = 0,
	\end{align*}
	by definition of $\psi^{(N)}$.

\smallskip

Let us turn to uniqueness. Let us consider another flat solution $\left(\left\{\widetilde{Y}^i,\widetilde{Z}^i\right\}_{1\leq i \leq N},\widetilde{K}^{(N)}\right)$. We have
\begin{equation*}
	\widetilde{Y}_t^i=U_t^i +\, \mathbb{E}\left[\psi_T^{(N)} + \widetilde{K}_T^{(N)}\,\middle|\,\mathscr{F}_t^{(N)}\right]-\widetilde{K}_t^{(N)}
\end{equation*}
and, since the constraint is satisfied, by definition of $\psi^{(N)}$, the supermartingale 
\begin{equation*}
	\mathbb{E}\left[\psi_T^{(N)} + \widetilde{K}_T^{(N)}\,\middle|\,\mathscr{F}_t^{(N)}\right]-\widetilde{K}_t^{(N)}
\end{equation*}
is bounded from below by the process $\psi^{(N)}_t$. Since $S$ is the Snell envelope of $\psi^{(N)}$, we have
\begin{equation*}
	\mathbb{E}\left[\psi_T^{(N)} + \widetilde{K}_T^{(N)}\,\middle|\,\mathscr{F}_t^{(N)}\right]-\widetilde{K}_t^{(N)} \geq S_t, \qquad \widetilde{Y}_t^i \geq Y^i_t.
\end{equation*}
Let us suppose that there exists $(i,t)$ such that $\mathbb{P}\left(\widetilde{Y}_{t}^i>Y_{t}^i\right) >0$. Let us consider the stopping time
\begin{equation*}
	\tau = \inf\left\{ u \geq t : \widetilde{Y}_{u}^i=Y_{u}^i\right\}.
\end{equation*}
 Then, on the set $\left\{\widetilde{Y}_{t}^i>Y_{t}^i\right\}$,  $T\geq \tau >t$ and $\widetilde{Y}_u^i>Y_u^i$ for $t\in[t,\tau)$. Therefore, on this set, since $h$ is increasing,
 \begin{equation*}
 	\sum_{j=1}^N{h(\widetilde{Y}_t^j)}>\sum_{j=1}^N{h(Y_t^j)}\geq 0
 \end{equation*} 
 and $d\widetilde{K}^{(N)}\equiv 0$ on $[t,\tau)$ due to the Skorokhod condition.

	We have, by definition of $\tau$,
	\begin{align*}
		Y_t^i - \widetilde{Y}_{t}^i & = Y_{\tau}^i - \widetilde{Y}_{\tau}^i -\int_{t}^{\tau} {\sum_{j=1}^N{\left(Z_u^{i,j}-\widetilde{Z}_u^{i,j} \right)\,dB_u^j}} + K_{\tau}^{(N)}-K_{t}^{(N)} - \left(\widetilde{K}_{\tau}^{(N)}-\widetilde{K}_{t}^{(N)}\right), \\
	  &= -\int_{t}^{\tau}{\sum_{j=1}^N{\left(Z_u^{i,j}-\widetilde{Z}_u^{i,j} \right)\,dB_u^j}} + K_{\tau}^{(N)}-K_{t}^{(N)} - \left(\widetilde{K}_{\tau}^{(N)}-\widetilde{K}_{t}^{(N)}\right),
	\end{align*}
	and, since on the set $\{\widetilde{Y}_{t}^i>Y_{t}^i\}$ $\widetilde{K}_{t}^{(N)}=\widetilde{K}_{\tau}^{(N)}$,
	\begin{align*}
	\left(Y_{t}^i - \widetilde{Y}_{t}^i\right)\mathbf{1}_{\widetilde{Y}_{t}^i>Y_{t}^i} 
	&= \left(K_{\tau}^{(N)}-K_{t}^{(N)}\right)\mathbf{1}_{\widetilde{Y}_{t}^i>Y_{t}^i} -\mathbf{1}_{\widetilde{Y}_{t}^i>Y_{t}^i}\int_{t}^{\tau}{\sum_{j=1}^N{\left(Z_u^{i,j}-\widetilde{Z}_u^{i,j} \right)\,dB_u^j}}.
	\end{align*}
	Taking the expectation, we get
	\begin{equation*}
		0> \e\left[\left(Y_{t}^i - \widetilde{Y}_{t}^i\right)\mathbf{1}_{\widetilde{Y}_{t}^i>Y_{t}^i}\right] = \mathbb{E}\left[\left(K_{\tau}^{(N)}-K_{t}^{(N)}\right)\,\mathbf{1}_{\widetilde{Y}_{t}^i>Y_{t}^i} \right] \geq 0,
	\end{equation*}
which is a contradiction. So $\widetilde{Y}^i=Y^i$ for all $1\leq i\leq N$. By uniqueness of the Doob-Meyer decomposition, it follows that $\widetilde{K}^{(N)}=K^{(N)}$  and $\widetilde{Z}=Z$.
\end{proof}

\begin{pr}
	\label{en:utile}
	Let $0\leq t\leq T$. For $1\leq i\leq N$,
	\begin{equation*}
		Y^i_s = \e\left[Y^i_t + \int_s^t f^i_u\, du \,\middle|\,\mathscr{F}_s^{(N)}\right] + R_s,\quad 0\leq s\leq t,
	\end{equation*}
	where $\{R_s\}_{}{0\leq s\leq t}$ is the Snell envelope of the process
	\begin{equation*}
		\phi_s^{(N)}=\inf\left\{x\geq 0 : \frac{1}{N}\, \sum_{i=1}^N h\left(x+\e\left[Y^i_t + \int_s^t f^i_u\, du \,\middle|\,\mathscr{F}_s^{(N)}\right]\right)\geq 0\right\}, \quad 0\leq s\leq t.
	\end{equation*}
\end{pr}

\begin{proof}
	Let us fix $0\leq t \leq T$ and $1\leq i\leq N$. For $s\leq t$, since $Y^i_t = U^i_t + S_t$,
	\begin{align*}
		\psi^{(N)}_s & = \inf\left\{x\geq 0 : \frac{1}{N}\, \sum_{i=1}^N h\left(x+\e\left[\xi^i + \int_s^T f^i_u\, du \,\middle|\,\mathscr{F}_s^{(N)}\right]\right)\geq 0\right\} \\
		& = \left(\inf\left\{x\in\rset : \frac{1}{N}\, \sum_{i=1}^N h\left(x+\e\left[\xi^i + \int_s^T f^i_u\, du \,\middle|\,\mathscr{F}_s^{(N)}\right]\right)\geq 0\right\}\right)_+ \\
		& = \left(\inf\left\{x\in\rset : \frac{1}{N}\, \sum_{i=1}^N h\left(x-\e\left[S_t\,|\, \mathscr{F}_s^{(N)}\right]+\e\left[Y^i_t + \int_s^t f^i_u\, du \,\middle|\,\mathscr{F}_s^{(N)}\right]\right)\geq 0\right\}\right)_+ \\
		& = \left(\e\left[S_t\,|\, \mathscr{F}_s^{(N)}\right] +\inf\left\{x\in\rset : \frac{1}{N}\, \sum_{i=1}^N h\left(x+\e\left[Y^i_t + \int_s^t f^i_u\, du \,\middle|\,\mathscr{F}_s^{(N)}\right]\right)\geq 0\right\}\right)_+.
	\end{align*}
	Since $S$ is a supermartingale, $S_s\geq \e\left[S_t\,|\, \mathscr{F}_s^{(N)}\right]$ and, taking into account the previous equality together with the fact that, for $a\geq 0$, $\left[(x+a)_+ - a\right]_+ = x_+$,
	\begin{equation*}
		S_s \geq \max\left(\e\left[S_t\,|\, \mathscr{F}_s^{(N)}\right], \psi^{(N)}_s\right) = \e\left[S_t\,|\, \mathscr{F}_s^{(N)}\right] + \phi^{(N)}_s.
	\end{equation*}
	It follows by definition of $R$ that, for $0\leq s \leq t$,
	\begin{equation*}
		S_s \geq \e\left[S_t\,|\, \mathscr{F}_s^{(N)}\right] + R_s \geq \e\left[S_t\,|\, \mathscr{F}_s^{(N)}\right] + \phi^{(N)}_s \geq \psi^{(N)}_s.
	\end{equation*}
	Since $S$ is the smallest supermartingale above $\psi^{(N)}$, we have actually, for $0\leq s\leq t$,
	\begin{equation*}
		S_s = \e\left[S_t\,|\, \mathscr{F}_s^{(N)}\right] + R_s.
	\end{equation*}
	As a byproduct,
	\begin{align*}
		Y^i_s & = \e\left[U^i_t + \int_s^t f^i_u\, du \,|\, \mathscr{F}_s^{(N)} \right] + S_s = \e\left[Y^i_t-S_t + \int_s^t f^i_u\, du \,|\, \mathscr{F}_s^{(N)} \right] + S_s, \\
		& = \e\left[Y^i_t+ \int_s^t f^i_u\, du \,|\, \mathscr{F}_s^{(N)} \right] + R_s.
	\end{align*}
\end{proof}

Let us end this section by an a priori estimate.

\begin{pr}\label{APEst}
There exists a constant $C$ independent of $N$ such that, for all $1\leq i\leq N$,
\[
\left\|Y^i\right\|^2_{\mathscr{S}^2}+\left\|Z^i\right\|^2_{\mathscr{M}^2}+\left\|K^{(N)}\right\|^2_{\mathscr{A}^2}\leq C\left( 1+ \mathbb{E}[\xi^2]+ \left\|f\right\|^2_{\mathscr{M}^2} \right).
\]
\end{pr}

\begin{proof}

Since $Y_t^i = U_t^i+S_t$ and $S_t=\esssup_{\tau\geq t}\,\mathbb{E}\left[\psi_{\tau}^{(N)}\,\middle|\,\mathscr{F}_t^{(N)}\right]$,
\begin{align*}
|Y_t^i| \leq \left|U_t^i\right| + {\esssup}_{\tau\geq t}\,\mathbb{E}\left[\left|\psi_{\tau}^{(N)}\right|\,\middle|\,\mathscr{F}_t^{(N)}\right] \leq \left|U_t^i\right| + \mathbb{E}\left[ \underset{s}{\sup}\,\left|\psi_s^{(N)}\right|\,\middle|\,\mathscr{F}_t^{(N)}\right] 
\end{align*}
and the estimate for $Y^i$ follows from Doob's inequality together with the bound for $\psi^{(N)}$ given by \eqref{eq:majpsi}.

Applying Ito's formula, we obtain
\begin{equation*}
	\left|Y_t^i\right|^2+\int_t^T{\sum_{j=1}^N{\left|Z_u^{i,j}\right|^2} \,du}= |\xi^i|^2 +2\int_t^T{Y_u^i \,f_u^i \, du}+2\int_t^T{Y_u^i  \,dK_u^{(N)}}-2\int_t^T{Y_u^i\,\sum_{j=1}^N{Z_u^{i,j} \,dB_u^j}}
\end{equation*}
and we get, for $0\leq t\leq T$,
\begin{multline}
	\label{eq:use}
	\mathbb{E}\left[ \left|Y_t^i\right|^2+\int_t^T{\left|Z_u^i\right|^2 du} \right] \\
	\leq \mathbb{E}[\xi^2]+ \mathbb{E}\left[\int_t^T{|f_u|^2 \,du}\right]+ (T-t+8)\, \mathbb{E}\left[\underset{s\in[t,T]}{\sup}\left|Y_s^i\right|^2\right]+ \frac{1}{8}\, \mathbb{E}\left[\left|K_T^{(N)}-K_t^{(N)}\right|^2\right].
\end{multline}
Since
\begin{equation*}
	K_T^{(N)}=Y_0^i-\xi^i-\int_0^T{f_u^i \,du}+\int_0^T{\sum_{j=1}^N}{Z_u^{i,j}dB_u^j},
\end{equation*}
we have, with the previous estimate,
\begin{align*}
\mathbb{E}\left[\left|K_T^{(N)}\right|^2\right]
&\leq 4\mathbb{E}\left[\left|Y_0^i\right|^2 +\int_0^T{\left|Z_u^i\right|^2 du}\right]+4\left( \mathbb{E}[\xi^2]+T\, \mathbb{E}\left[\int_0^T{|f_u|^2 \,du}\right] \right)\\
&\leq \frac{1}{2}\, \mathbb{E}\left[\left|K_T^{(N)}\right|^2\right] + 4\left( 2\mathbb{E}[\xi^2]+(1+T) \left\|f\right\|^2_{\mathscr{M}^2} \right)+ 4(T+8)\, \mathbb{E}\left[\underset{t}{\sup}|Y_t^i|^2\right]
\end{align*}
which gives the bound for $K^{(N)}$. Coming back to~\eqref{eq:use}, we get the estimate for $Z^i$ and the result.
\end{proof}

\begin{rem}\label{UBK}
In what follows, we will also need an upper bound for $|K_T|^2$. It is however easier to obtain as for $\mathbb{E}\left[\left|K_T^{(N)}\right|^2\right]$ since we have $K_T=\underset{t}{\sup}\,\psi_t$ and some Lipschitz property for $\psi$ (see \cite{MRBSDE}) :

\begin{flalign*} &\left|K_T\right|\leq \underset{t}{\sup}\left|\psi_t\right| \leq  \frac{M}{m}\,\underset{t}{\sup}\, \mathbb{E}\left[\,\left|\mathbb{E}\left[\xi+\int_t^T{f_u\, du}\,\middle|\,\mathscr{F}_t\right]-\xi\right|\,\right]\leq \frac{M}{m} \left(2\mathbb{E}\left[|\xi|\right]+\mathbb{E}\left[\int_0^T{|f_u| \,du}\right]\right).&\\
&\text{so there exists \,} \tilde{C}(m,M,T) \text{\, such that \,}  \left|K_T\right|^2\leq \tilde{C}(m,M,T)\left( \mathbb{E}[\xi^2]+  \left\|f\right\|^2_{\mathscr{M}^2} \right).&
\end{flalign*}
\end{rem}

\bigskip

\section{Propagation of chaos : general reflexion}

In this section, we deal with the case where the driver depends on $y$ but does not depend on $z$. Equation \eqref{eqPS} rewrites in this case
\begin{equation}\label{sansZ}
\begin{dcases}{} Y_t^i=\theta^i+\int_t^T{f^i(u,Y_u^i) \,du}-\int_t^T{\sum_{j=1}^N{Z_u^{i,j} \,dB_u^j}}+K_T^{(N)}-K_t^{(N)} \quad \forall\, 1\leq i\leq N\\ \frac{1}{N}\sum_{i=1}^N{h(Y_t^i)}\geq 0 \end{dcases} \; 0\leq t\leq T
\end{equation}

\begin{pr}\label{en:eusz}
The reflected BSDE~\eqref{sansZ} has a unique square integrable flat solution.
\end{pr}

\begin{proof}
	We use a fixed point argument. Let us introduce the map $\Gamma$ from $\mathscr{S}^2(\mathbb{R})^N$ into itself defined by $Y=\Gamma(P)$ where $(Y,Z,K)$ stands for the unique square integrable flat solution to 
\[
\begin{dcases} Y_t^i=\theta^i+\int_t^T{f^i(u,P_u^i) \,du}-\int_t^T{\sum_{j=1}^N{Z_u^{i,j} \,dB_u^j}}+K_T^{(N)}-K_t^{(N)} \quad \forall\, 1\leq i\leq N\\ \frac{1}{N}\sum_{i=1}^N{h(Y_t^i)}\geq 0 \end{dcases} \; 0\leq t\leq T.
	\]

Let $\left\{Y^i\right\}_{1\leq i\leq N}=\Gamma\left(\left\{P^i\right\}_{1\leq i\leq N}\right), \left\{\widetilde{Y}^i\right\}_{1\leq i\leq N}=\Gamma\left(\left\{\widetilde{P}^i\right\}_{1\leq i\leq N}\right)$ and denote by $\Delta\cdot$ the corresponding differences. We have
	\begin{align*}
		\left|\Delta Y^i_t\right| & \leq \left|\Delta U^i_t\right| + |\Delta S_t| \leq \lambda T \, \mathbb{E}\left[\sup_s \left|\Delta P_s^i\right| \,\middle|\, \mathscr{F}^{(N)}_t\right] + \mathbb{E}\left[\sup_s \left|\Delta \psi_s^{(N)}\right| \,\middle|\, \mathscr{F}^{(N)}_t\right].
	\end{align*}
	From Doob's inequality, we get
	\begin{align*}
		\mathbb{E}\left[\sup_{0\leq t\leq T} \left|\Delta Y^i_t\right|^2\right] & \leq 8\lambda^2 T^2 \mathbb{E}\left[\sup_{0\leq t\leq T} \left|\Delta P^i_t\right|^2\right] + 8 \,\mathbb{E}\left[\sup_{0\leq t\leq T} \left|\Delta \psi^{(N)}_t\right|^2 \right],
	\end{align*}
	and using \eqref{eq:Llip}, we get
	\begin{align*}
		\mathbb{E}\left[\sup_{0\leq t\leq T} \left|\Delta Y^i_t\right|^2\right] & \leq 8 \lambda^2 T^2 \mathbb{E}\left[\sup_{0\leq s\leq T} \left|\Delta P^i_t\right|^2\right] + 8 \left(\dfrac{M}{m}\right)^2 \mathbb{E}\left[\sup_{0\leq t\leq T} \left( \frac{1}{N} \sum_{j=1}^N  |\Delta U^j_{t}|  \right)^2 \right].
	\end{align*}
	But, we have, since $f$ is Lipschitz,
	\begin{equation*}
		\frac{1}{N}\sum_{j=1}^N  |\Delta U^j_{t}| \leq \lambda T \, \mathbb{E}\left[ \frac{1}{N}\sum_{j=1}^N \sup_s \left|\Delta P_s^j\right| \,\middle|\, \mathscr{F}^{(N)}_t\right],
	\end{equation*}
	and Doob's and H\"older's inequalities lead to
	\begin{equation*}
		\mathbb{E}\left[\sup_{0\leq t\leq T} \left|\Delta Y^i_t\right|^2\right] \leq 8 \lambda^2 T^2 \mathbb{E}\left[\sup_{0\leq t\leq T} \left|\Delta P^i_t\right|^2\right] + 32 \lambda^2 T^2 \left(\dfrac{M}{m}\right)^2 \mathbb{E}\left[ \frac{1}{N}\sum_{j=1}^N \sup_{0\leq t\leq T} |\Delta P^j_t |^2 \right].
	\end{equation*}
	Summing these inequalities gives
	\begin{equation*}
		\mathbb{E}\left[\frac{1}{N}\sum_{i=1}^N \sup_{0\leq t\leq T} \left|\Delta Y^i_t\right|^2\right]  \leq 8 \lambda^2 T^2 \left( 1+ 4\left(\dfrac{M}{m}\right)^2 \right) \mathbb{E}\left[ \frac{1}{N}\sum_{i=1}^N \sup_{0\leq t\leq T} \left|\Delta P^i_{t}\right|^2 \right].
	\end{equation*}
	
	It follows that $\Gamma$ has a unique fixed point in $\mathscr{S}^2(\mathbb{R})^N$ as soon as $T$ is small enough: there exists a unique $\{Y^i\}_{1\leq i\leq N}$ solving \eqref{sansZ} for some $\left(\{Z^i\}_{1\leq i\leq N},K^{(N)}\right)\in \mathscr{M}^2(\mathbb{R}^N)^N\times \mathscr{A}^2(\mathbb{R})$. Since $\{Y^i\}_{1\leq i\leq N}$ is unique, Itô's formula shows that $\{Z^i\}_{1\leq i\leq N}$ is also unique and we deduce finally that $K^{(N)}$ is also unique.
	
For larger values of $T$, let $\varepsilon$ be such that $\lambda^2 \varepsilon^2 \left( 1+ 4\left(\dfrac{M}{m}\right)^2 \right) \leq \dfrac{1}{16}$ and let us pick an integer $r$ such that $T/r < \varepsilon$. 

For $k=0,\ldots,r$, $T_k = kT/r$. Denote, for $k=r,\ldots,1$, let $\left(\{Y^{i,k},Z^{i, k}\}_{1\leq i\leq N},K^{(N), k}\right)$ be the unique triple constructed on $[T_{k-1} , T_{k}]$ with $K^{(N),k}_{T_{k-1}}=0$. The triple $\left(\{Y^i,Z^i\}_{1\leq i\leq N},K^{(N)}\right)$ defined by
\[
Y_t^i = Y^{i,k}_t, \quad Z_t^i = Z^{i, k}_t, \quad K_t^{(N)} = K_t^{(N),k}+\sum_{\ell<k} K_{T_{\ell}}^{(N),\ell} \quad T_{k-1}\leq t \leq T_k 
\]
is the unique flat solution to equation \eqref{sansZ}.
\end{proof}

\begin{pr}
	\label{en:estsz}
	There exists a constant $C$ independent of $N$ such that, for all $1\leq i\leq N$,
	\[
	\left\|Y^i\right\|^2_{\mathscr{S}^2}+\left\|Z^i\right\|^2_{\mathscr{M}^2}+\left\|K^{(N)}\right\|^2_{\mathscr{A}^2}\leq C\left( 1+ \mathbb{E}[\xi^2] + \mathbb{E}\left[\int_0^T |f(t,0)|^2 dt\right] \right).
	\]
\end{pr}

\begin{proof}
	In this proof, $C$ denotes a constant independent of $N$ which may change from line to line.
	
	Let $(T_k)_{0\leq k\leq r}$ be a subdivision of $[0,T]$ with $\max_{1\leq k\leq r} (T_k-T_{k-1})=\pi$. We set $I_k = [T_{k-1}, T_k]$. By Proposition~\ref{en:utile}, for $1\leq k\leq r$ and $t\in I_k$,
	\begin{equation*}
		Y^i_t = \e\left[Y^i_{T_k} + \int_t^{T_k} f^i(u,Y^i_u)\, du \,\middle|\,\mathscr{F}_t^{(N)}\right] + R_t,
	\end{equation*}
	where $\{R_t\}_{T_{k-1}\leq t\leq T_{k}}$ is the Snell envelope of the process
	\begin{equation*}
		\phi_t^{(N)}=\inf\left\{x\geq 0 : \frac{1}{N}\, \sum_{i=1}^N h\left(x+\e\left[Y^i_{T_k} + \int_t^{T_k} f^i(u,Y^i_u)\, du \,\middle|\,\mathscr{F}_t^{(N)}\right]\right)\geq 0\right\}.
	\end{equation*}
	Doing the same computation as in the proof of Proposition~\ref{en:eusz}, we get
	\begin{equation}\label{eq:i}
		\mathbb{E}\left[\sup_{t\in I_k} \left|Y^i_t\right|^2\right] \leq  C A(i,k) + C\pi\, \mathbb{E}\left[\sup_{t\in I_k} \left|Y^i_t\right|^2\right] +  C\, \frac{1}{N}\sum_{j=1}^N \left(A(j,k) + \pi\, \mathbb{E}\left[\sup_{t\in I_k} \left|Y^j_t\right|^2\right] \right),
	\end{equation}
	where, for $1\leq i\leq N$ and $1\leq k \leq r$,
	\begin{equation*}
		A(j,k) = 1 + \e\left[|Y^j_{T_k}|^2 + \int_{I_k} |f^j(s,0)|^2\, ds \right].
	\end{equation*}
	Summing these inequalities gives
	\begin{equation*}
		\mathbb{E}\left[\frac{1}{N}\sum_{i=1}^N \sup_{t\in I_k} \left|Y^i_t\right|^2\right]  \leq C\, \frac{1}{N}\sum_{j=1}^N \left(A(j,k) + \pi\, \mathbb{E}\left[\sup_{t\in I_k} \left|Y^j_t\right|^2\right] \right).
	\end{equation*}
	Let us choose $\pi$ small enough to get, for $1\leq k \leq r$, 
	\begin{equation*}
		\mathbb{E}\left[\frac{1}{N}\sum_{i=1}^N \sup_{t\in I_k} \left|Y^i_t\right|^2\right]  \leq C\, \frac{1}{N}\sum_{j=1}^N A(j,k).
	\end{equation*}
	Let us observe that
	\begin{align*}
		A(j,r) & = 1+ \e\left[ |\xi^j|^2 + \int_{I_k} |f^j(s,0)|^2\, ds \right], \\
		A(j,r) & \leq 1+ \e\left[ \sup_{t\in I_{k+1}}|Y^j_t|^2 + \int_{I_k} |f^j(s,0)|^2\, ds \right], \quad 1\leq k \leq r-1.
	\end{align*}
	Thus, for any constant $\alpha >0$,
	\begin{multline*}
		\sum_{k=1}^r \alpha^k\, \mathbb{E}\left[\frac{1}{N}\sum_{i=1}^N \sup_{t\in I_k} \left|Y^i_t\right|^2\right] \\
		\leq C\alpha^r \left(1+ \e\left[|\xi|^2 + \int_0^T |f(s,0)|^2 ds\right]\right) + \frac{C}{\alpha} \sum_{k=1}^r \alpha^k\, \mathbb{E}\left[\frac{1}{N}\sum_{i=1}^N \sup_{t\in I_k} \left|Y^i_t\right|^2\right],
	\end{multline*}
	and choosing $\alpha > C$, we get
	\begin{equation*}
		\mathbb{E}\left[\frac{1}{N}\sum_{i=1}^N \sup_{0\leq t\leq T} \left|Y^i_t\right|^2\right] \leq C\left(1+ \e\left[|\xi|^2 + \int_0^T |f(s,0)|^2 ds\right]\right).
	\end{equation*}
	With the help of this inequality, we can go back to \eqref{eq:i} and do the same computation, to get, for $0\leq i\leq N$,
	\begin{equation*}
		\mathbb{E}\left[\sup_{0\leq t\leq T} \left|Y^i_t\right|^2\right] \leq C\left(1+ \e\left[|\xi|^2 + \int_0^T |f(s,0)|^2 ds\right]\right).
	\end{equation*}
	
	We conclude the proof exactly as in the proof of Proposition~\ref{APEst}.
\end{proof}

Let us recall that,  if $\xi = G\left(\{B_t\}_{0\leq t\leq T}\right)$ and $f(t,y,z)=F\left(t,\{B_{s\wedge t}\}_{0\leq s\leq T},y,z\right)$ for some measurable $G$ and $F$, we took
\begin{equation*}
	\xi^i = G\left(\{B^i_t\}_{0\leq t\leq T}\right), \qquad f^i(t,y,z) = F\left(t,\{B^i_{s\wedge t}\}_{0\leq s\leq T},y,z\right).
\end{equation*}
Let us consider $\left(\overline{Y}^{\,i}, \overline{Z}^{\,i}, K\right)$ independent copies of $(Y,Z,K)$ i.e $\left(\overline{Y}^{\,i}, \overline{Z}^{\,i}, K\right)$ is the flat deterministic solution to
\begin{equation*}
	\overline{Y}^{\,i}_t = \xi^i + \int_t^T f^i\left(u,\overline{Y}^{\,i}_u\right)du - \int_t^T \overline{Z}^{\,i}_u dB^i_u + (K_T-K_t), \quad 0\leq t\leq T,
\end{equation*}
with $\mathbb{E}\left[h\left(\overline{Y}^{\,i}_t\right)\right]\geq 0$.

\begin{thm}\label{Conv}
Let us set, for $1\leq i\leq N\,,\, \Delta Y^i:=Y^i-\overline{Y}^{\,i}, \Delta K:= K^{(N)}-K$ and $\Delta Z^i:= Z^i_t -\overline{Z}^{\,i}\,e_i$ where $(e_1,\ldots,e_N)$ stand for the canonical basis in $\mathbb{R}^N$.
\begin{enumerate}
\item
If $h$ is of class $\mathscr{C}^2$ with bounded derivatives and $\underset{t}{\sup}\, \mathbb{E}\left[\left|\overline{Z}_t^{\,1}\right|^4\right]<\infty$ then
\begin{center}
$\left\|\Delta Y^i\right\|^2_{\mathscr{S}^2}=\mathcal{O}(N^{-1})\,,\, \left\|\Delta Z^i\right\|^2_{\mathscr{M}^2}=\mathcal{O}(N^{-1/2})$ and $\left\|\Delta K\right\|^2_{\mathscr{A}^2}=\mathcal{O}(N^{-1/2})$.
\end{center}
\item
If $\xi\in L^p,\, f(.,0)\in\mathscr{M}^p$ and $\underset{t}{\sup}\, \mathbb{E}\left[\left|\overline{Z}_t^{\,1}\right|^p\right]<\infty$\, for some $p>4$ then 
\begin{center}
$\left\|\Delta Y^i\right\|^2_{\mathscr{S}^2}=\mathcal{O}(N^{-1/2})\,,\, \left\|\Delta Z^i\right\|^2_{\mathscr{M}^2}=\mathcal{O}(N^{-1/4})$ and $\left\|\Delta K\right\|^2_{\mathscr{A}^2}=\mathcal{O}(N^{-1/4})$.
\end{center}
\end{enumerate} 
\end{thm}

The following lemma gives sufficient conditions on the terminal condition and the driver for the extra assumption on $Z$ to be satisfied.

\begin{lem}\label{en:CN}
Suppose for a given $p\geq 2,$ $\xi\in L^p, f(.,0)\in \mathscr{M}^p$ and $f$ $\lambda$-Lipschitz with respect to $y$ uniformly in time. Suppose also $f$ continuously differentiable in $y$ with uniformly bounded derivative and $\xi$ and $f(.,y)$ Malliavin differentiable for each $y$ with
\begin{enumerate}
	\item $\underset{\theta}{\sup}\, \mathbb{E}\left[ \left|D_{\theta}\xi\right|^p \right] <\infty$
	\item $D_{\theta} f(t,y)$ is $K_{\theta}$-Lipschitz continuous in $y$ uniformly in time  and 
\[ 
\underset{\theta}{\sup}\, \left\| D_{\theta} f(.,0) \right\|_{\mathscr{M}^p} <\infty, \qquad \underset{\theta}{\sup}\, K_{\theta} <\infty.
 \]
\end{enumerate}

Then,  $\sup_t\e\left[|Z_t|^p\right] < \infty$.
\end{lem}

\begin{exam}
	Let us consider the Markovian framework 
	\begin{equation*}
		\xi = g(X_T), \quad f(s,y)= F(X_T,y),
	\end{equation*}
	where
	\begin{equation*}
		X_t = x_0 + \int_0^t b(X_r)\, dr + \int_0^t \sigma(X_r)\, dB_r, \quad 0\leq t\leq T.
	\end{equation*}
	The the assumptions of the previous lemme are satisfied when $b$, $\sigma$, $g$ and $F$ are continuously differentiable with $\sigma$, $\partial_x b$, $\partial_x \sigma$ and $\partial_y F$ bounded and $\partial_x g$ and $\partial_x F$ with polynomial growth.
\end{exam}

\begin{proof}[Proof of Theorem~\ref{Conv}] Let us recall that for all $1\leq i\leq N,$
\begin{align*}
Y_t^i &= U_t^i + S_t \text{\, with \,} S_t=\underset{\tau \text{\,s.t\,}\geq t}{\esssup\,}\mathbb{E}\left[\psi_{\tau}^{(N)}\,\middle|\,\mathscr{F}_t^{(N)}\right],\\
\overline{Y}_t^{\,i} &= \overline{U}_t^{\,i} + R_t \text{\, with \,} R_t=\underset{s\geq t}{\sup}\,\psi_s= \underset{\tau \text{\,s.t\,}\geq t}{\esssup\,}\mathbb{E}\left[\psi_{\tau}\,\middle|\,\mathscr{F}_t^{(N)}\right].
\end{align*}
We consider also
\[\overline{\psi}_t^{\,(N)}= \inf\left\{ x\geq 0 : \dfrac{1}{N}\sum_{j=1}^N{h(x+\overline{U}_t^j)}\geq 0\right\}.\]
Since the Brownian motion are independent,
\begin{align*}
	\overline{U}_t^{\,i} = \mathbb{E}\left[ \xi^i + \int_t^T{f^i(u,\overline{Y}_u^{\,i}) \,du} \,\middle|\, \mathscr{F}^i_t \right] & = \mathbb{E}\left[ \xi^i + \int_t^T{f^i(u,\overline{Y}_u^{\,i}) \,du} \,\middle|\, \mathscr{F}^{(N)}_t \right], \\
	& := \mathbb{E}\left[ \xi^i + \int_t^T{\overline{f}_u^{\,i}\,du} \,\middle|\, \mathscr{F}^{(N)}_t \right],
\end{align*}
and, we have
\begin{equation*}
	U_t^{\,i}= \mathbb{E}\left[ \xi^i + \int_t^T{f^i(u,Y_u^{\,i}) \,du} \,\middle|\, \mathscr{F}^{(N)}_t \right].
\end{equation*}

$\bullet$ \textbf{Step 1} We have, for $t\geq r$,
\begin{multline*}
\left|\Delta Y^i_t\right|
 \leq \left|\Delta U^i_t\right| + |S_t-R_t| \leq \lambda  \mathbb{E}\left[ \int_t^T{\left|\Delta Y^i_u\right| du} \,\middle|\, \mathscr{F}^{(N)}_t\right] + \mathbb{E}\left[\sup_{t\leq s\leq T} \left|\psi^{(N)}_s - \psi_s\right| \,\middle|\, \mathscr{F}^{(N)}_t\right], \\
 \leq \lambda  \mathbb{E}\left[ \int_t^T{\left|\Delta Y^i_u\right| du} \,\middle|\, \mathscr{F}^{(N)}_t\right] + \mathbb{E}\left[\sup_{r\leq s\leq T} \left|\Delta \psi^{(N)}_s\right| \,\middle|\, \mathscr{F}^{(N)}_t\right] + \mathbb{E}\left[\sup_{0\leq s\leq T} \left|\overline{\psi}^{\,(N)}_s -\psi_s\right| \,\middle|\, \mathscr{F}^{(N)}_t\right],
\end{multline*}
and by Doob's inequality
\begin{equation*}
\mathbb{E}\left[\sup_{r\leq t\leq T}\left|\Delta Y^i_t\right|^2\right] \leq 8\lambda^2 T \mathbb{E}\left[\int_r^T{\left|\Delta Y^i_u\right|^2 du}\right] + 16 \mathbb{E}\left[\sup_{r\leq t\leq T} \left|\Delta \psi^{(N)}_t\right|^2\right] + 16\left\|\overline{\psi}^{\,(N)}-\psi\right\|^2_{\mathscr{S}^2}.
\end{equation*}

On the other hand, using \eqref{eq:Llip}, when $t\geq r$,
\begin{equation*}
\left|\Delta \psi^{(N)}_t\right| \leq \frac{M}{m} \frac{1}{N} \sum_{j=1}^N{\left|\Delta U^j_t\right|} \leq \lambda \, \frac{M}{m} \mathbb{E}\left[ \frac{1}{N} \sum_{j=1}^N \int_r^T \left|\Delta Y^j_u\right| du \,\middle|\, \mathscr{F}^{(N)}_t\right].
\end{equation*}
Using Doob's inequality and then H\"older's inequality, we get, for all $0\leq r\leq T$,
\begin{multline}\label{eq:presque}
\mathbb{E}\left[\sup_{r\leq t\leq T}\left|\Delta Y^i_t\right|^2\right] \\ \leq 8\lambda ^2 T \left( \mathbb{E}\left[\int_r^T{\left|\Delta Y^i_u\right|^2 du}\right] + 8\left(\frac{M}{m}\right)^2 \mathbb{E}\left[\frac{1}{N}\sum_{j=1}^N \int_r^T{\left|\Delta Y^j_u\right|^2 du} \right] \right)+ 16 \left\| \overline{\psi}^{\,(N)} -\psi \right\|^2_{\mathscr{S}^2}.
\end{multline}
Summing these inequalities, we obtain
\begin{align*}
\mathbb{E}\left[\frac{1}{N} \sum_{i=1}^N\sup_{r\leq t\leq T}\left|\Delta Y^i_t\right|^2\right] & \leq 8\lambda^2 T \left(1+8\left(\frac{M}{m}\right)^2\right) \int_r^T{\mathbb{E}\left[ \frac{1}{N}\sum_{j=1}^N{\left|\Delta Y^j_u\right|^2} \right] du} + 16 \left\| \overline{\psi}^{\,(N)} -\psi \right\|^2_{\mathscr{S}^2}\\
		 & \leq 8\lambda^2 T \left(1+8\left(\frac{M}{m}\right)^2\right) \int_r^T{\mathbb{E}\left[ \frac{1}{N}\sum_{j=1}^N{\sup_{u\leq s\leq T}\left|\Delta Y^j_s\right|^2} \right] du}  + 16 \left\| \overline{\psi}^{\,(N)} -\psi \right\|^2_{\mathscr{S}^2} 
\end{align*}
and Gronwall's Lemma gives
\begin{equation*}
\mathbb{E}\left[\frac{1}{N} \sum_{i=1}^N{\sup_{0\leq t\leq T}\left|\Delta Y^i_t\right|^2}\right] \leq 16 \exp\left(8\lambda^2 T \left(1+8\left(\frac{M}{m}\right)^2\right)\right) \left\|\overline{\psi}^{\,(N)}-\psi\right\|^2_{\mathscr{S}^2}.
\end{equation*}
	
	Coming back to the estimate~\eqref{eq:presque}, we finally deduce that
\begin{equation*}
\left\|\Delta Y^i\right\|^2_{\mathscr{S}^2} \leq C(\lambda,m,M,T) \left\|\overline{\psi}^{\,(N)}-\psi\right\|^2_{\mathscr{S}^2}.
\end{equation*}

$\bullet$ \textbf{Step 2} For all $0\leq t\leq T$, denote $\nu_t$ the common law of the random variables $\left\{\overline{U}_t^{\,i}\right\}_{1\leq i\leq N}$ and their empirical law $\displaystyle \nu_t^{(N)}:= \frac{1}{N}\sum_{i=1}^N{\delta_{\overline{U}_t^{\,i}}}$.

Let us define $H:(x,\mu)\in\mathbb{R}\times\mathcal{M}_1\mapsto \int{h(x+y)\,\mu(dy)}$. For each probability measure $\mu$, $x\longmapsto H(x,\mu)$ is nondecreasing and bi-Lipschitz withe same constants as $h$. Let us also introduce 
\begin{equation*}
	\psi_t^*= \inf\left\{ x\in\mathbb{R} : \mathbb{E}\left[h(x+\overline{U}_t^{\,i})\right]\geq 0\right\}
\end{equation*}
and ${\overline{\psi}_t^{\,(N)}}^*$ defined in the same way. Since $h$ is continuous, one has 
\begin{equation*}
	H(\psi_t^*,\nu_t)=H({\overline{\psi}_t^{\,(N)}}^*,\nu_t^{(N)})=0.
\end{equation*} 

Of course, $\left|\psi_t-\overline{\psi}_t^{\,(N)}\right|\leq \left|\psi_t^*-{\overline{\psi}_t^{\,(N)}}^*\right|$ so that
\begin{equation}\label{HLip}
\begin{split}
\left\|\overline{\psi}^{\,(N)}-\psi\right\|^2_{\mathscr{S}^2} &\leq \dfrac{1}{m^2} \mathbb{E}\left[\underset{0\leq t\leq T}{\sup}\,\left|H({\overline{\psi}^{\,(N)}}^*,\nu_t^{(N)})-H(\psi_t^*,\nu_t^{(N)})\right| ^2\right],\\
&\leq \dfrac{1}{m^2} \mathbb{E}\left[\underset{0\leq t\leq T}{\sup}\,\left|H(\psi_t^*,\nu_t^{(N)})-H(\psi_t^*,\nu_t)\right| ^2\right].
\end{split}
\end{equation}

\textbf{1. The smooth case.}\quad 
Let us start by the case where $h$ is smooth. Since $\sup_t \e\left[\left|\overline Z^1_t\right|^2\right]$ is finite, it is not hard to check, as done in~\cite{MRSDE}, that $t\longmapsto \psi^*_t$ is locally Lipschitz.

Set $\displaystyle \Delta H:=H(\psi_t^*,\nu_t^{(N)})-H(\psi_t^*,\nu_t) =\frac{1}{N}\sum_{i=1}^N\left\{h(V_t^i)- \mathbb{E}\left[h(V_t^i)\right]\right\}$ where $V_t^i:=\psi_t^*+\overline{U}_t^{\,i}$.\\ It comes from Ito's formula that
\begin{align*}
\Delta H &=\frac{1}{N}\sum_{i=1}^N\left\{h(V_T^i)-\mathbb{E}\left[h(V_T^i)\right]\right\} -\int_t^T\frac{1}{N}\sum_{i=1}^N\left\{h'(V_u^i)\left(\Psi_u-\overline{f}_u^{\,i}\right)-\mathbb{E}\left[h'(V_u^i)\left(\Psi_u-\overline{f}_u^{\,i}\right)\right]\right\}du \\
&- \frac{1}{N}\int_t^T{\sum_{i=1}^N{ h'(V_u^i)\overline{Z}_u^{\,i}\, dB_u^i}} - \frac{1}{2}\int_t^T\frac{1}{N}\sum_{i=1}^N\left\{h''(V_u^i)\left|\overline{Z}_u^{\,i}\right|^2 - \mathbb{E}\left[h''(V_u^i)\left|\overline{Z}_u^{\,i}\right|^2\right]\right\}du
\end{align*}
with $\Psi$ the Radon-Nikodym derivative of $\psi^*$. Combined with \eqref{HLip}, it follows
\begin{align*}
&\left\|\overline{\psi}^{\,(N)}-\psi\right\|^2_{\mathscr{S}^2}\leq \frac{4}{N m^2} \mathbb{V}\left[h(V_T^1)\right]+\frac{4T}{N m^2}\int_0^T{\mathbb{V}\left[h'(V_u^1)\left(\Psi_u-\overline{f}_u^{\,i}\right)\right] du} +\frac{16}{N m^2}\int_0^T{\mathbb{E}\left[\left|h'(V_s^1)\overline{Z}_s^{\,1}\right|^2\right] ds}\\
&+ \frac{T}{N m^2}\int_0^T{\mathbb{V}\left[h''(V_u^i)\left|\overline{Z}_u^{\,i}\right|^2\right] du}  \\
&\leq \frac{4}{N} \left(\frac{M}{m}\right)^2\left( \mathbb{E}\left[\left|V_T^1\right|^2\right]+T \mathbb{E}\left[\int_0^T{\left|\overline{f}_u^{\,1}\right|^2 du}\right]\right) + \frac{16}{N}\left(\frac{M}{m}\right)^2\left\|\overline{Z}^{\,1}\right\|^2_{\mathscr{M}^2} + \frac{T}{N}\left(\frac{\left\|h''\right\|_{\infty}}{m}\right)^2 \mathbb{E}\left[\int_0^T{\left|\overline{Z}_u^{\,1}\right|^4 du}\right]\\
&\leq \frac{8}{N} \left(\frac{M}{m}\right)^2\left( \mathbb{E}\left[\xi^2\right] + \left|\psi_T^*\right|^2+T \left|\left|f(.,0)\right|\right|^2_{\mathscr{M}^2}\right) + \frac{8}{N} \left(\frac{M}{m}\right)^2 \lambda^2 T^2 \left\|\overline{Y}^{\,1}\right\|^2_{\mathscr{S}^2} + \frac{16}{N}\left(\frac{M}{m}\right)^2\left\|\overline{Z}^{\,1}\right\|^2_{\mathscr{M}^2}  \\
&+ \frac{T^2}{N}\left(\frac{\left\|h''\right\|_{\infty}}{m}\right)^2 \underset{t\in[0,T]}{\sup} \mathbb{E}\left[\left|\overline{Z}_t^{\,1}\right|^4\right] \leq \frac{C_1}{N} \quad \text{with\,} C_1 \text{\, depending on all parameters}.
\end{align*}

\textbf{2. The general case.}\quad
 Since $h$ is $M$-Lipschitz$,$ we deduce from \eqref{HLip} that 
\begin{equation*}
	\left\|\overline{\psi}^{\,(N)}-\psi\right\|^2_{\mathscr{S}^2} \leq \left(\frac{M}{m}\right)^2 \mathbb{E}\left[\underset{t}{\sup}\, W_1^2\left(\nu_t^{(N)},\nu_t\right)\right].
\end{equation*}
The right hand side of the previous inequality can be estimated by Theorem 10.2.7 of \cite{RR98}. We will use here a better bound, proved in \cite{MRSDE} based on recent results by Fournier and Guillin in \cite{FG15}. We have:
\begin{equation*}
	\left\|\overline{\psi}^{\,(N)}-\psi\right\|^2_{\mathscr{S}^2} \leq \frac{C_2}{\sqrt{N}}
\end{equation*}
with $C_2$ depending on $p$ and all parameters.

Indeed we can apply this result, since, for all $0\leq s\leq t\leq T$, we have
\[
\overline{U}_s^{\,i}-\overline{U}_t^{\,i} = \int_s^t{f\left(u,\overline{Y}_u^{\,i}\right) du} - \int_s^t{\overline{Z}_u^{\,i} dB_u^i},
\]
and, under our assumptions, for $1\leq q\leq p$,
\begin{align*}
\mathbb{E}\left[\left|\overline{U}_t^{\,i}-\overline{U}_s^{\,i}\right|^q\right] 
&\leq C_q(m,M,T)\left[ \mathbb{E}\left[|\xi|^q\right] + \left\|f(.,0)\right\|^q_{\mathscr{M}^q} + \underset{0 \leq u\leq T}{\sup} \mathbb{E}\left[\left|\overline{Z}_u^{\,1}\right|^q\right] \right] \left|t-s\right|^{q/2}.
\end{align*}

$\bullet$ \textbf{Step 3} $(\Delta Y^i,\Delta Z^i,\Delta K)$ verifies \[ \Delta Y_t^i=\psi_T^{(N)} +  \int_t^T{\left(f^i\left(u,Y_u^i\right)-f^i\left(u,\overline{Y}_u^{\,i}\right)\right) du} -\int_t^T{\sum_{j=1}^N{\Delta Z_u^{i,j}dB_u^j}} +\,\Delta K_T-\Delta K_t \quad \forall t\in[0,T]\]

Use Ito's formula to obtain that
\begin{align*}
\mathbb{E}\left[\int_0^T{\sum_{j=1}^N{\left|\Delta Z_u^{i,j}\right|^2}du}\right] \leq&\, \mathbb{E}\left[\left|\psi_T^{(N)}\right|^2 + 2\int_0^T{\Delta Y_u^i \left(f^i\left(u,Y_u^i\right)-f^i\left(u,\overline{Y}_u^{\,i}\right)\right) du} + 2 \int_0^T{\Delta Y_u^i d\Delta K_u}\right]\\
\leq&\, \mathbb{E}\left[\left|\psi_T^{(N)}\right|^2\right] + 2\lambda T \left\|\Delta Y^i\right\|^2_{\mathscr{S}^2} + 2\sqrt{2}\, \left\|\Delta Y^i\right\|_{\mathscr{S}^2} \left(\mathbb{E}\left[\left|K_T^{(N)}\right|^2 \right]+\left|K_T\right|^2 \right)^{1/2}
\end{align*}
Since we know from Proposition~\ref{en:estsz} that
\begin{equation*}
	\left|K_T\right|^2 + \sup_{N} \mathbb{E}\left[\left|K_T^{(N)}\right|^2 \right] < +\infty,
\end{equation*}
we deduce the rate of convergence
\[ 
\left\|\Delta Z^i\right\|^2_{\mathscr{M}^2}= O\left(\left\|\Delta Y^i\right\|_{\mathscr{S}^2}\right). 
\]

Finally, let us write
\begin{align*}
\Delta K_t &= \Delta Y_0^i - \Delta Y_t^i - \int_0^t{\left(f^i\left(u,Y_u^i\right)-f^i\left(u,\overline{Y}_u^{\,i}\right)\right) \,du} + \int_0^t{\sum_{j=1}^N{\Delta Z_u^{i,j}dB_u^j}} \quad \forall t\in[0,T]\\
\end{align*}
to get
\begin{align*}
 \left\|\Delta K\right\|^2_{\mathscr{A}^2} &\leq 3\left\|\Delta Y^i\right\|^2_{\mathscr{S}^2}+ 12\left\|\Delta Z^i\right\|^2_{\mathscr{M}^2} + 3T\, \mathbb{E}\left[\int_0^T{\left|f^i\left(u,Y_u^i\right)-f^i\left(u,\overline{Y}_u^{\,i}\right)\right|^2 du}\right]\\
&\leq 3\left(1+\lambda^2 T^2\right)\left\|\Delta Y^i\right\|^2_{\mathscr{S}^2}+ 12\left\|\Delta Z^i\right\|^2_{\mathscr{M}^2} = O\left(\left\|\Delta Y^i\right\|_{\mathscr{S}^2}\right).
\end{align*}

This ends the proof of our main result.

\end{proof}

\begin{proof}[Proof of Corollary~\ref{en:CN}]
We give the elements of the proof in the case $p=2$ since it is then easy to generalize to the case $p\geq 2,$ as mentioned in \cite{ElKar2} of which Proposition 5.3 is the basis of our result.

\smallskip

\noindent Indeed$,$ we verify the hypothesis for this Proposition and since $K$ is deterministic$,$ we can extend its conclusion to our MRBSDE. Therefore$,$ $Y$ and $Z$ are Malliavin differentiable$,$ there derivatives solve
\[\begin{dcases}
D_{\theta} Y_t = D_{\theta} Z_t = 0 \qquad \qquad \forall 0\leq t<\theta \leq T  \\
D_{\theta} Y_t = D_{\theta} \xi + \int_t^T{\left[\partial_y f(u,Y_u)D_{\theta} Y_u + D_{\theta} f(u,Y_u)\right]\,du} - \int_t^T{D_{\theta} Z_u \,dB_u} \quad \forall \theta \leq t \leq T
\end{dcases}\]
and a version of $Z$ is given by $\left\{ D_t Y_t , 0\leq t\leq T \right\}$.

\medskip

It comes classically that there exists $C>0$ depending only on $T$ such that
\begin{align*}
	\underset{\theta}{\sup}\,\left\| D_{\theta} Y \right\|^2_{\mathscr{S}^2} &\leq C \left( \underset{\theta}{\sup}\,\mathbb{E}\left[\left|D_{\theta} \xi\right|^2\right] + \underset{\theta}{\sup}\,\left\| D_{\theta} f(.,Y) \right\|^2_{\mathscr{M}^2} \right) \\
&\leq C \left( \underset{\theta}{\sup}\,\mathbb{E}\left[\left|D_{\theta} \xi\right|^2\right] + 2\,\underset{\theta}{\sup}\,\left\| D_{\theta} f(.,0) \right\|^2_{\mathscr{M}^2} + 2T \left\|Y\right\|^2_{\mathscr{S}^2} \underset{\theta}{\sup}\,\left|K_{\theta}\right|^2 \right) <\infty.
\end{align*}
This inequality allows to conclude since 
\begin{equation*}
	\underset{t\in[0,T]}{\sup} \mathbb{E}\left[\left|Z_t\right|^2\right] \leq \underset{t}{\sup}\, \mathbb{E}\left[\underset{s}{\sup} \left|D_t Y_s\right|^2\right] \leq \underset{\theta}{\sup}\,\left\| D_{\theta} Y \right\|^2_{\mathscr{S}^2}.
\end{equation*}
\end{proof}

\section{The case of linear reflexion}

In this subsection, we are concerned with the case of linear reflexions. In this framework, we can deal with generators that depend both on $y$ and $z$. Our additional assumption is the following:
\begin{center}
$(\widetilde{H}_h)$ The function $h$ is given by $h(x)=ax+b$ for some $a >0$ and $b\in\mathbb{R}$.
\end{center}
 
\begin{pr}\label{en:lineu}
Let $(H_{\xi})$, $(H_f)$ and $(\widetilde{H}_h)$ hold. The reflected BSDE~\eqref{eqPS} has a unique square integrable flat solution. $\left(\left\{Y^i,Z^i\right\}_{1\leq i\leq N},K^{(N)}\right)$. Moreover,
\begin{equation*}
\sup_{N\geq 1} \mathbb{E}\left[ \left| K^{(N)}_T\right|^2 \right] < \infty.
\end{equation*}
\end{pr}
 
\begin{proof}
The existence and uniqueness of $\left(\left\{Y^i,Z^i\right\}_{1\leq i\leq N},K^{(N)}\right)$ solution to~\eqref{eqPS} result from \cite{GPP96} since 
\begin{equation*}
	\left\{y\in\mathbb{R}^N : \sum_{i=1}^N h(y^i) \geq 0\right\}
\end{equation*} 
is a convex set in $\mathbb{R}^N$.
	
For $\alpha=\dfrac{3}{2}+2\lambda+3\lambda^2$, we get from Itô's formula, for $1\leq i\leq N$, noting $c=b/a$,
\begin{multline*}
\mathbb{E}\left[e^{\alpha t} \left|Y^i_t+c\right|^2 + \frac{1}{2} \int_t^T{e^{\alpha u} \left(\left|Y^i_u+c\right|^2 + \left|Z^i_u\right|^2\right) du} \right] \\
	\leq \mathbb{E}\left[e^{\alpha T} \left|\xi^i+\psi^{(N)}_T+c\right|^2 + \int_0^T{e^{\alpha u}\left|f^i(u,-c,0)\right|^2 du} +2 \int_t^T{e^{\alpha u} \left(Y^i_u+c\right) dK^{(N)}_u} \right]\\
		 = \mathbb{E}\left[e^{\alpha T} \left|\xi^i+\psi^{(N)}_T+c\right|^2 + \int_0^T{e^{\alpha u}\left|f^i(u,-c,0)\right|^2 du} + \frac{2}{a} \int_t^T{e^{\alpha u} h(Y^i_u) dK^{(N)}_u} \right].
\end{multline*}
Summing these inequalities and using the Skorokhod condition, we get the estimate, since $\{\xi^i\}_{1\leq i\leq N}$ and $\{f^i(.,0,0)\}_{1\leq i\leq N}$ are IID copies of $\xi$ and $f(.,0,0)$,
\begin{equation*}
\sup_{0\leq t\leq T} \frac{1}{N}\sum_{i=1}^N{\mathbb{E}\left[\left|Y^i_t+c\right|^2 + \int_t^T{\left(\left|Y^i_u+c\right|^2 + \left|Z^i_u\right|^2\right) du}\right]} \leq C \left(1+\mathbb{E}\left[|\xi|^2+\left|\psi^{(N)}_T\right|^2+ \int_0^T |f(u,0,0)|^2 du\right]\right).
\end{equation*}
Let us observe that
\begin{equation*}
\psi^{(N)}_T = \left(\frac{1}{N}\sum_{i=1}^N \xi^i + \frac{b}{a}\right)^{-},
\end{equation*}
to get, using again the fact that $\{\xi^i\}_{1\leq i\leq N}$  are IID copies of $\xi$,
\begin{equation*}
\sup_{0\leq t\leq T} \frac{1}{N}\sum_{i=1}^N{\mathbb{E}\left[\left|Y^i_t+c\right|^2 + \int_t^T{\left(\left|Y^i_u+c\right|^2 + \left|Z^i_u\right|^2\right) du}\right]} \leq C \left(1+\mathbb{E}\left[|\xi|^2+ \int_0^T{|f(u,0,0)|^2 du}\right]\right).
\end{equation*}

For each $i$, we have
\begin{equation*}
K^{(N)}_T = Y^i_0 - Y_T^i - \int_0^T{f^i\left(u,Y^i_u,Z^i_u\right)\, du} + \int_0^T{\sum_{j=1}^N{Z^{i,j}_u dB^{j}_u}}
\end{equation*}
from which we deduce that, since $\{f^i(.,0,0)\}_{1\leq i\leq N}$ are IID copies of $f(.,0,0)$,
\begin{equation*}
\mathbb{E}\left[\left|K^{(N)}_T\right|^2\right] \leq C \mathbb{E}\left[\left|Y^i_0\right|^2 + \left|Y^i_T\right|^2+ \int_0^T{\left|Z^i_u\right|^2 du} + \int_0^T{\left|f(u,0,0)\right|^2 du}\right].
\end{equation*}
The result follows by taking the arithmetic mean over $i$ of these inequalities.
\end{proof}
 
\begin{rem}
It follows from the construction in Section \ref{Tpswcd}, that $Y^i_t = U^i_t + S_t$ where
\begin{equation*}
U^i_t = \mathbb{E}\left[\xi^i + \int_t^T{f^i\left(u,Y^i_u,Z^{i,i}_u\right) du} \,\middle|\, \mathscr{F}_t^{(N)} \right]
\end{equation*}
and $S_t$ is the Snell envelope of
\begin{equation*}
\psi^{(N)}_t = \left(\frac{1}{N} \sum_{i=1}^N{U^i_t} + \frac{b}{a}\right)^-.
\end{equation*}
\end{rem}
 
As before, let us consider $(\overline{Y}^{\,i}, \widehat Z^i, K)$ independent copies of $(Y,Z,K)$ i.e.
\begin{equation*}
\overline{Y}_t^{\,i}=\xi^i + \int_t^T{f^i(u,\overline{Y}_u^{\,i},\widehat{Z}_u^{\,i})\,du} - \int_t^T{\widehat{Z}_u^{\,i} \,dB_u^i}+K_T-K_t,
\end{equation*}
with $\mathbb{E}\left[h(\overline{Y}^{\,i}_t)\right] \geq 0$ and the Skorokhod condition. Let us set $\overline{Z}^{\,i,j} = \widehat Z^i \mathds{1}_{i=j}$. With this notation, the previous equation rewrites
\begin{equation*}
\overline{Y}_t^{\,i}=\xi^i + \int_t^T{f^i(u,\overline{Y}_u^{\,i},\overline{Z}_u^{\,i,i})\,du} - \int_t^T{ \sum_{i=1}^N \overline{Z}_u^{\,i,j} \,dB_u^j}+K_T-K_t.
\end{equation*}
and $\overline{Y}_t^{\,i} = \overline{U}_t^{\,i} + R_t$ where
\begin{equation*}
\overline{U}_t^{\,i}= \mathbb{E}\left[ \xi^i + \int_t^T{f^i(u,\overline{Y}_u^{\,i},\overline{Z}_u^{\,i,i}) \,du} \,\middle|\, \mathscr{F}^{(N)}_t \right] \; \text{are IID random variables}
\end{equation*}
and $R_t$ is the Snell envelope of
\begin{equation*}
\psi_t=\left( \mathbb{E}\left[\overline{U}_t^{\,i}\right]+\frac{b}{a} \right)^-.
\end{equation*}
 
\begin{thm}
Let us denote $\Delta Y^i:=Y^i-\overline{Y}^{\,i}$, $\Delta Z^i := Z^i-\overline{Z}^{\,i}$ and $\Delta K:= K^{(N)}-K$. Then, there exists a constant independent of $N$, such that, for $1\leq i\leq N$,
\begin{equation*}
\mathbb{E}\left[\sup_{0\leq t\leq T} \left|\Delta Y^i_t\right|^2 + \int_0^T{\left|\Delta Z^i_u\right|^2 du} + \sup_{0\leq t\leq T} \left|\Delta K_t\right|^2\right] \leq \frac{C}{\sqrt N}.
\end{equation*}
\end{thm}
 
\begin{proof}
For all $1\leq i\leq N$, the triple $(\Delta Y^i,\Delta Z^i,\Delta K)$ solves the BSDE on $[0,T]$
\begin{equation*}
\Delta Y_t^i=\psi_T^{(N)} +  \int_t^T{\left(f^i\left(u,Y_u^i,Z_u^{i,i}\right)-f^i\left(u,\overline{Y}_u^{\,i},\overline{Z}_u^{\,i,i}\right)\right) du} -\int_t^T{\sum_{j=1}^N{\Delta Z_u^{i,j}dB_u^j}} +\,\Delta K_T-\Delta K_t.
\end{equation*}

$\bullet$ \textbf{Step 1.}\quad For $\alpha\geq 2\lambda^2+2\lambda+\dfrac{1}{2}$, we get from Itô's formula, for $1\leq i\leq N$,
\begin{align*}
\frac{1}{2} \, \mathbb{E}\left[ \int_0^T{e^{\alpha u} \left(\left|\Delta Y^i_u\right|^2 + \left|\Delta Z^i_u\right|^2\right) du} \right] & \leq \mathbb{E}\left[e^{\alpha T}\left|\psi^{(N)}_T\right|^2 +2 \int_0^T{ e^{\alpha u} \Delta Y^i_u d\Delta K_u} \right], \\
& = \mathbb{E}\left[e^{\alpha T}\left|\psi^{(N)}_T\right|^2 + \frac{2}{a} \int_0^T{e^{\alpha u} \left(h\left(Y^i_u\right) - h\left(\overline{Y}^{\,i}_u\right)\right) d\Delta K_u} \right].
\end{align*}
Summing these inequalities and using the Skorokhod conditions together with constraint, we get
\begin{equation*}
\frac{1}{N} \sum_{i=1}^N{\mathbb{E}\left[\int_0^T{\left( \left|\Delta Y_u^i\right|^2 + \left|\Delta Z_u^i\right|^2\right) du}\right]}
\leq 2e^{\alpha T}\mathbb{E}\left[\left|\psi_T^{(N)}\right|^2\right] + \frac{4}{aN}\sum_{i=1}^N{\mathbb{E}\left[\int_0^T{e^{\alpha u}\left(-h\left(\overline{Y}_u^{\,i}\right)\right)dK^{(N)}_u}\right]},
\end{equation*}
and since $\mathbb{E}\left[h\left(\overline{Y}^{\,i}_t\right)\right]$ is nonnegative,
\begin{multline*}
\frac{1}{N} \sum_{i=1}^N{\mathbb{E}\left[\int_0^T{\left( \left|\Delta Y_u^i\right|^2 + \left|\Delta Z_u^i\right|^2\right) du}\right]} \\
\leq 2e^{\alpha T}\mathbb{E}\left[\left|\psi_T^{(N)}\right|^2\right] + \frac{4}{aN}\sum_{i=1}^N{\mathbb{E}\left[\int_0^T{e^{\alpha u}\left(\mathbb{E}\left[ h\left(\overline{Y}_u^{\,i}\right) \right] -h\left(\overline{Y}_u^{\,i}\right)\right) dK_u^{(N)}}\right]} \\
	  = 2e^{\alpha T}\mathbb{E}\left[\left|\psi_T^{(N)}\right|^2\right] + 4\mathbb{E}\left[\int_0^T{e^{\alpha u}\, \frac{1}{N}\sum_{i=1}^N{\left\{\mathbb{E}\left[\overline{U}_u^{\,i}\right] -\overline{U}_u^{\,i}\right\}} \,dK_u^{(N)}}\right]\\
	 \leq 2e^{\alpha T}\mathbb{E}\left[\left|\psi_T^{(N)}\right|^2\right] + 4e^{\alpha T} \mathbb{E}\left[\left|K_T^{(N)}\right|^2\right]^{1/2} \mathbb{E}\left[\underset{t}{\sup}\left|\frac{1}{N}\sum_{i=1}^N{\left\{\overline{U}_t^{\,i} - \mathbb{E}\left[\overline{U}_t^{\,i}\right] \right\}}\right|^2\right]^{1/2}.
\end{multline*}

But it follows from Proposition~\ref{en:lineu}, that $\mathbb{E}\left[\left|K_T^{(N)}\right|^2\right]$ is bounded uniformly in $N,$ and we have
\begin{multline*}
\frac{1}{N}\sum_{i=1}^N{\left\{\overline{U}_t^{\,i} - \mathbb{E}\left[\overline{U}_t^{\,i}\right] \right\}}  \\
= \frac{1}{N}\sum_{i=1}^N{\left\{\overline{U}_T^{\,i} - \mathbb{E}\left[\overline{U}_T^{\,i}\right] \right\}} + \int_t^T{ \frac{1}{N}\sum_{i=1}^N{\left\{\overline{f}_u^{\,i} - \mathbb{E}\left[\overline{f}_u^{\,i}\right] \right\}} \,du} - \frac{1}{N}\int_t^T{\sum_{i=1}^N{\overline{Z}_u^{\,i,i}} dB_u^i}
\end{multline*} 
which implies
\begin{align*}
\mathbb{E}\left[\underset{t}{\sup}\left|\frac{1}{N}\sum_{i=1}^N{\left\{\overline{U}_t^{\,i} - \mathbb{E}\left[\overline{U}_t^{\,i}\right] \right\}}\right|^2\right] &\leq \frac{3}{N}\mathbb{V}\left[\overline{U}_T^{\,1}\right] + \frac{3T}{N}\int_0^T{\mathbb{V}\left[\overline{f}_u^{\,1}\right]du} + \frac{12}{N}\mathbb{E}\left[\int_0^T{\left|\overline{Z}_u^{\,1}\right|^2du}\right]\\
&\leq \frac{3}{N} \mathbb{E}\left[\xi^2\right] + \frac{3T}{N}\mathbb{E}\left[\int_0^T{\left|f(u,Y_u,Z_u)\right|^2 du}\right] + \frac{12}{N}\left\|Z\right\|^2_{\mathscr{M}^2}.
\end{align*}
It follows that, for some constant $C$ independent of $N$,
\begin{equation*}
\frac{1}{N} \sum_{i=1}^N{\mathbb{E}\left[\int_0^T{\left( \left|\Delta Y_u^i\right|^2 + \left|\Delta Z_u^i\right|^2\right) du}\right]}
\leq 2e^{\alpha T}\mathbb{E}\left[\left|\psi_T^{(N)}\right|^2\right] + \frac{C}{\sqrt{N}}.
\end{equation*}

Due to the symmetry of $K^{(N)}$ with respect to the particles, for any permutation $\sigma$ of $\{1,\ldots, N\}$, the law of $\left(\Delta Y^i, \left\{\Delta Z^{i,j}\right\}_{1\leq j\leq N}\right)$ is the same as the law of $\left(\Delta Y^{\sigma(i)}, \left\{\Delta Z^{\sigma(i),\sigma(j)}\right\}_{1\leq j\leq N}\right)$. In particular, the law of $\left(\Delta Y^i, \left|\Delta Z^i\right|\right)$ is independent of $i$. Thus it follows,
\begin{equation*}
\mathbb{E}\left[\int_0^T{\left(\left|\Delta Y^i_u\right|^2 + \left|\Delta Z^i_u\right|^2\right) du}\right] \leq 2e^{\alpha T}\mathbb{E}\left[\left|\psi_T^{(N)}\right|^2\right] + \frac{C}{\sqrt{N}}.
\end{equation*}

$\bullet$ \textbf{Step 2.}\quad  
Introduce
\begin{equation*}
\overline{\psi}_t^{\,(N)} = \left( \frac{1}{N}\sum_{i=1}^N{\overline{U}_t^{\,i}}+\frac{b}{a} \right)^-, \quad 0\leq t\leq T.
\end{equation*}
Notice that, since $U_T^i = \overline{U}_T^{\,i}=\xi^i$, $\psi_T^{(N)} = \overline{\psi}_T^{\,(N)}$ which can be rewritten as $\overline{\psi}_T^{\,(N)} - \psi_T$ since $\mathbb{E}[h(\xi)]\geq 0$.

We obtain, with classical results on IID random variables, 
\begin{align*}
\left\|\overline{\psi}^{\,(N)}-\psi\right\|^2_{\mathscr{S}^2} &\leq \frac{12}{N} \mathbb{E}\left[\xi^2\right] + \frac{15T}{N}\mathbb{E}\left[\int_0^T{\left|f(u,Y_u,Z_u)\right|^2\,du}\right].
\end{align*}
It follows that
\begin{equation*}
\mathbb{E}\left[\int_0^T{\left(\left|\Delta Y^i_u\right|^2 + \left|\Delta Z^i_u\right|^2\right) du}\right] \leq \frac{C}{\sqrt N}.
\end{equation*}

$\bullet$ \textbf{Step 3.}\quad 
Let us recall that, the processes $Y^i$ and $\overline{Y}^{\,i}$ are given by
\begin{align*}
Y_t^i &= U_t^i + S_t \text{\, with \,} S_t=\underset{\tau \text{\,s.t\,}\geq t}{\esssup\,}\mathbb{E}\left[\psi_{\tau}^{(N)}\,\middle|\,\mathscr{F}_t^{(N)}\right],\\
\overline{Y}_t^{\,i} &= \overline{U}_t^{\,i} + R_t \text{\, with \,} R_t=\underset{s\geq t}{\sup}\,\psi_s= \underset{\tau \text{\,s.t\,}\geq t}{\esssup\,}\mathbb{E}\left[\psi_{\tau}\,\middle|\,\mathscr{F}_t^{(N)}\right].
\end{align*}
and remark that
\begin{equation*}
\left|S_t-R_t\right|\leq \underset{\tau \text{\,s.t\,}\geq t}{\esssup\,}\mathbb{E}\left[\left|\psi_{\tau}-\psi_{\tau}^{(N)}\right|\,\middle|\,\mathscr{F}_t^{(N)}\right]\leq \mathbb{E}\left[\underset{s\geq t}{\sup}\,\left|\psi_s-\psi_s^{(N)}\right|\,\middle|\,\mathscr{F}_t^{(N)}\right].
\end{equation*}

It follows from the previous inequality that
\begin{align*}
\left\|\Delta Y^i\right\|^2_{\mathscr{S}^2} &\leq 2\left\|U^i - \overline{U}^{\,i}\right\|^2_{\mathscr{S}^2} + 2\left\|S-R\right\|^2_{\mathscr{S}^2} \leq 2\left\|U^i - \overline{U}^{\,i}\right\|^2_{\mathscr{S}^2} + 8\left\|\psi^{(N)}-\psi\right\|^2_{\mathscr{S}^2}\\
&\leq 2\left\|U^i - \overline{U}^{\,i}\right\|^2_{\mathscr{S}^2} + 16\left\|\psi^{(N)}-\overline{\psi}^{\,(N)}\right\|^2_{\mathscr{S}^2} + 16\left\|\overline{\psi}^{\,(N)}-\psi\right\|^2_{\mathscr{S}^2}.
\end{align*}
Using \eqref{eq:Llip}, since the particles are exchangeable, we obtain
\begin{align*}
\left\|\Delta Y^i\right\|^2_{\mathscr{S}^2} &\leq 18\left\|U^i - \overline{U}^{\,i}\right\|^2_{\mathscr{S}^2} + 16\left\|\overline{\psi}^{\,(N)}-\psi\right\|^2_{\mathscr{S}^2}, \\
&\leq 72 \,T\, \mathbb{E}\left[\int_0^T{\left|f^i(u,Y_u^i,Z_u^{i,i})-f^i(u,\overline{Y}_u^{\,i},\overline{Z}_u^{\,i})\right|^2 du}\right] + 16\left\|\overline{\psi}^{\,(N)}-\psi\right\|^2_{\mathscr{S}^2} \\
&\leq 144\lambda^2 T \mathbb{E}\left[\int_0^T{\left(\left|\Delta Y^i_u\right|^2 + \left|\Delta Z^i_u\right|^2\right) du}\right] + 16\left\|\overline{\psi}^{\,(N)}-\psi\right\|^2_{\mathscr{S}^2}
\end{align*}
and finally $\displaystyle \left\|\Delta Y^i\right\|^2_{\mathscr{S}^2} + \left\|\Delta Z^i \right\|^2_{\mathscr{M}^2} \leq \dfrac{C}{\sqrt{N}}$.

\bigskip

$\bullet$ \textbf{Step 4.}\quad
Finally, let us write
\[
\Delta K_t = \Delta Y_0^i - \Delta Y_t^i - \int_0^t{f^i(u,Y_u^i,Z_u^{i,i})-f^i(u,\overline{Y}_u^{\,i},\overline{Z}_u^{\,i,i})\,du} + \int_0^t{\sum_{j=1}^N{\Delta Z_u^{i,j}dB_u^j}}, \quad 0\leq t\leq T,
\]
to get
\begin{align*}
\mathbb{E}\left[\sup_{0\leq t\leq T}\left|\Delta K_t\right|^2\right] &\leq 3\left\|\Delta Y^i\right\|^2_{\mathscr{S}^2}+ 12\left\|\Delta Z^i\right\|^2_{\mathscr{M}^2} + 3T\, \mathbb{E}\left[\int_0^T{\left|f^i(u,Y_u^i,Z_u^{i,i})-f^i(u,\overline{Y}_u^{\,i},\overline{Z}_u^{\,i})\right|^2 du}\right],
\end{align*}
from which we deduce
\begin{equation*}
	\mathbb{E}\left[\sup_{0\leq t\leq T}\left|\Delta K_t\right|^2\right] \leq \frac{C}{\sqrt N}.
\end{equation*}

 \end{proof}


%

\end{document}